\documentclass[12pt]{article}
\usepackage{amsmath}
\usepackage{amsfonts}
\usepackage{amsthm}
\usepackage{amssymb, setspace}
\usepackage{color,graphicx,epsfig,geometry,hyperref,fancyhdr}
\usepackage[T1]{fontenc}
\usepackage{float}
\usepackage{subfig}
\usepackage{bbm}
\usepackage{mathrsfs,fleqn}
\newtheorem{theorem}{Theorem}[section]

\newtheorem{assumption}{Assumption}[section]

\newcommand{\N}{\mathbb{N}}
\usepackage{algorithm2e}

\newcommand{\R}{\mathbb{R}}
\newcommand{\C}{\mathbb{C}}

\newcommand{\grad}{\nabla}

\graphicspath{{Pics/}}

\begin{document}

\begin{flushleft}
\Large 
\noindent{\bf \Large Approximation of the inverse scattering Steklov eigenvalues and the inverse spectral problem}
\end{flushleft}

\vspace{0.2in}

{\bf  \large Isaac Harris}\\
\indent {\small Department of Mathematics, Purdue University, West Lafayette, IN 47907 }\\
\indent {\small Email: \texttt{harri814@purdue.edu}}\\

\begin{abstract}
\noindent In this paper, we consider the numerical approximation of the Steklov eigenvalue problem that arises in inverse acoustic scattering. The underlying scattering problem is for an inhomogeneous isotropic medium. These eigenvalues have been proposed to be used as a target signature since they can be recovered from the scattering data. A Galerkin method is studied where the basis functions are the Neumann eigenfunctions of the Laplacian. Error estimates for the eigenvalues and eigenfunctions are proven by appealing to Weyl's Law. We will test this method against separation of variables in order to validate the theoretical convergence. We also consider the inverse spectral problem of estimating/recovering the refractive index from the knowledge of the Steklov eigenvalues. Since the eigenvalues are monotone with respect to a real-valued refractive index implies that they can be used for non-destructive testing. Some numerical examples are provided for the inverse spectral problem. 
\end{abstract}



\section{Introduction}
In this manuscript, we investigate the numerical approximation of a non-selfadjoint Steklov eigenvalue problem that arises in inverse acoustic scattering as well as the inverse spectral problem of estimating the material properties from the knowledge of the Steklov eigenvalues. 
A similar eigenvalue problem has been analyzed for the electromagnetic scattering problem in \cite{MaxwellStek}. The numerical method employed here is a Galerkin method where the basis functions are finitely many Neumann eigenfunctions of the Laplacian. In \cite{team-pancho} we see that the Neumann eigenfunctions of the Laplacian form a basis for the Sobolev space $H^1(D)$. Our convergence analysis of the Galerkin method will use the Weyl's asymptotic estimate for the Neumann eigenvalues. In \cite{ziteHarris} a similar method was used to approximate the zero-index transmission eigenvalues with a conductivity condition where finitely many Dirichlet eigenfunctions of the Laplacian are used as the approximation space. We will also numerically investigate the inverse spectral problem of estimating the refractive index from the Steklov eigenvalues. In our experiments we will see that the average value of the refractive index can be recovered numerically. 

The Steklov eigenvalue problem we consider here is associated with the direct scattering problem: find the total field $u \in H^1_{loc}(\R^d)$ for $d=2,3$ such that  
\begin{align}
\Delta u +k^2 n u=0  \quad \textrm{ in } \,\,\, \R^d \label{direct1}
\end{align}
with $u=u^s+u^i$. The incident field is given by $u^i=\text{e}^{\text{i} k x \cdot \hat{y} }$ where the incident direction $\hat{y}$ is a point on the unit circle/sphere. Here, we let $n \in L^{\infty} (\R^d)$ denote the refractive index with supp$(n-1)=\Omega$. We assume that the scatterer $\Omega \subset \R^d$ is a bounded simply connected open set with {\color{black}Lipschitz boundary}. The scattered field $u^s$ satisfies the Sommerfeld radiation condition 
\begin{align}
\lim\limits_{|x| \rightarrow \infty} |x|^{(d-1)/2} \left( \frac{\partial u^s}{\partial |x|} -\text{i}k u^s \right)=0 \label{src}
\end{align}
which is satisfied uniformly with respect to $\hat{x}=x/|x|$. Therefore, the  scattered field $u^s$ solving \eqref{direct1} and \eqref{src} has the asymptotic expansion as  $|x| \to \infty$ (see for e.g. \cite{cakoni-colton})
$$u^s(x,\hat{y})=\frac{\text{e}^{\text{i}k|x|}}{|x|^{(d-1)/2}} \left\{u^{\infty}(\hat{x}, \hat{y} ) + \mathcal{O} \left( \frac{1}{|x|}\right) \right\}$$
where $u^\infty$ denotes the {\color{black}`measured'} far-field pattern. Now, we define the corresponding far-field operator ${F}: L^2(\mathbb{S}) \longmapsto  L^2(\mathbb{S})$ by 
$$Fg(\hat{x})=\int_{\mathbb{S}} u^\infty(\hat{x},\hat{y}) g(\hat{y}) \, \text{d}s(\hat{y})$$
such that  $\mathbb{S}$ denotes the boundary of the unit circle/sphere. 
By only using the far-field operator $F$ one derives the associated {\color{black}transmission} eigenvalue problem (see for e.g. \cite{TE-book}). This is a non self-adjoint and nonlinear eigenvalues problem which makes the computation of these eigenvalues difficult. Therefore, in \cite{Aniso-stekloff} the authors augmented the inverse scattering problem with an auxiliary scattered field which leads to a linear eigenvalue problem. 

Now, just as in \cite{Aniso-stekloff} we define an auxiliary total field ${\color{black} u_\lambda \in H^1_{loc}(\R^d \setminus \overline{D})}$ with Im$(\lambda)\geq 0$ satisfying the system 
\begin{align}
\Delta u_\lambda +k^2 u_\lambda=0  \quad \textrm{ in } \,\,\, \R^d \setminus \overline{D}  \quad \text{and} \quad {\partial_\nu u_\lambda} + \lambda u_\lambda =0  \,\, \textrm{ on } \,\, \partial D \label{direct2} 
\end{align}
{\color{black}where $\nu$ is the unit outward normal vector on $\partial D$.} The auxiliary total field is given by $u_\lambda=u_\lambda^s+u^i$ and the auxiliary scattered field $u_\lambda^s$ also satisfies the Sommerfeld radiation condition \eqref{src}. Here, the region $D$ is taken to be any bounded simply connected open set with a {\color{black}$C^2-$boundary} such that $\Omega \subseteq D$. Similarly, the auxiliary scattered field $u_\lambda^s$ gives rise to the auxiliary far-field operator ${F}_\lambda: L^2(\mathbb{S}) \longmapsto  L^2(\mathbb{S})$ given by 
$$F_\lambda g(\hat{x})=\int_{\mathbb{S}} u_\lambda^\infty(\hat{x},\hat{y}) g(\hat{y}) \, \text{d}s(\hat{y}).$$
It is shown in \cite{Aniso-stekloff} that the modified far-field operator ${F} -{F}_\lambda$ is injective with a dense range if and only if $\lambda \in \C$ is not a Steklov eigenvalue for the scattering problem \eqref{direct1}. In \cite{Stek-invscat} it is shown that the knowledge of the modified far-field operator ${F} -{F}_\lambda$ can be used to recover the Steklov eigenvalues. Since $F$ is given by physical measurements and ${F}_\lambda$ can be computed numerically/analytically gives that the Steklov eigenvalues can be determined by the measurements. In \cite{Stek-invscat,Stekinvproblem} it is shown that the largest positive Steklov eigenvalue depends monotonically on a real-valued refractive index. This implies that the eigenvalue can be used as a target signature. Using this fact we will show that the Steklov eigenvalues can estimate a real-valued $n$.

We now define the inverse scattering Steklov eigenvalue problem associated with  \eqref{direct1}. These are defined as the values $\lambda \in \C$ with Im$(\lambda)\geq 0$ such that there is a nontrivial solution $w \in H^1(D)$ satisfying 
\begin{align}
\Delta w+k^2 n w = 0  \,\, \textrm{ in } \,\, {D} \quad \text{and} \quad {\partial_\nu w} + \lambda w =0  \,\, \textrm{ on } \,\, \partial D. \label{eig-problem} 
\end{align}
Recall, that $n \in L^{\infty}(D)$ such that supp$(n-1)=\Omega$  where the scatterer $\Omega \subseteq D$. Here $k>0$ denotes the wavenumber and for the non-selfadjoint case of an absorbing medium (i.e. $n$ is complex-valued) we have that 
$$n = n_{\text{R}} + \text{i}\frac{ n_{\text{I}} }{ k} \quad \text{with} \quad n_{\text{R}}>0 \,\, \text{and} \,\, n_{\text{I}} \geq 0.$$
This eigenvalue problem was introduced and studied in \cite{Aniso-stekloff,Stek-invscat} to overcome the shortcomings of the transmission eigenvalue problem that is obtained by only considering the injectivity of the far-field operator ${F}$. See \cite{gpinvtev} for a numerical method with the transmission eigenvalues for the inverse spectral problem. In \cite{SteknumericsSun} a continuous finite element method with a spectral indicator is used to approximate the Steklov eigenvalues. In \cite{StekDG} a discontinuous finite element method is used as an approximation scheme for this problem.  See \cite{spec-stek,VStek,SStek} for applications of other Galerkin methods applied to the selfadjoint Steklov eigenvalue problems. We also mention that this idea of augmenting the far-field operator $F$ by subtracting an auxiliary far-field operator has been employed in \cite{zite,Stektraceclass} to obtain new eigenvalue problems associated with  the scattering problem \eqref{direct1}.

The remainder of the paper is ordered as follows. We begin our investigation in the next section by defining the associated source problem for the Steklov eigenvalue problem. Next, we consider the approximation properties for the Neumann-Galerkin method's approximation space. Here we take our basis to be finitely many Neumann eigenfunctions for the Laplacian. Then we will study the convergence and prove error estimates for computing the Steklov eigenvalue and eigenfunctions. We will then provide some numerical examples in two dimensions for various scatterers. This will show that the proposed approximation is effective for computing the eigenvalues for a modest size discretized system. The numerical examples are given when $D$ is the unit circle but one can alternatively take a rectangular shaped domain. Lastly, we consider the inverse spectral problem of estimating the refractive index from the knowledge of the eigenvalues. 

\section{The Steklov Eigenvalue Problem}\label{problem-statement}
In this section, we will consider the variational formulation of the inverse scattering Steklov eigenvalue problem \eqref{eig-problem}. The analysis here will be used to prove the convergence of our approximation method. To begin, recall that the Sobolev space
$$H^1(D) =\big\{ \varphi \in L^2(D) \, : \, \partial_{x_i} \varphi  \in L^2(D) \,\,  \text{ for } \,\,  i =1, \cdots , d \big\}.$$
Now, by appealing to Green's First Theorem it is clear that the variational formulation of \eqref{eig-problem} is given by find $w \in H^1(D)$ such that 
\begin{align}
a(w,\varphi)= - \lambda b(w,\varphi)  \quad \text{ for all } \,\,\, \varphi \in H^1(D). \label{varform} 
\end{align} 
The bounded sesquilinear forms are defined by 
\begin{align}
a(w, \varphi)= \int\limits_{D} \grad w \cdot \grad \overline{\varphi} -k^2 n w \overline{\varphi} \, \text{d}x \quad \text{and } \quad 
b( w, \varphi)= \int\limits_{\partial D} {w} {\overline{\varphi}} \, \text{d}s. \label{forms}
\end{align} 
Since, the eigenfunction $w$ is assumed to be nontrivial we will assume that it is normalized with $\| w \|_{L^2(\partial D)}=1$. Note that $w \neq 0$ a.e. on $\partial D$ due to the impedance condition in \eqref{eig-problem}. Indeed, if not $w$ would have zero Cauchy data on $\partial D$ which would require $w=0$ by Green's Representation Theorem(see for e.g. \cite{coltonkress}). 

As in \cite{SteknumericsSun,StekDG} we will now define the associated Neumann-to-Dirichlet (NtD) operator for the source problem associated with \eqref{varform}. To this end, define the source problem: find $w \in H^1(D)$ such that for any $f \in L^2(\partial D)$
\begin{align}
a(w,\varphi)= b(f,\varphi)  \quad \text{ for all } \,\,\, \varphi \in H^1(D). \label{source} 
\end{align} 
It is clear that $w \in H^1(D)$ satisfies the boundary value problem 
$$\Delta w+k^2 n w = 0  \,\, \textrm{ in } \,\, {D} \quad \text{and} \quad {\partial_\nu w} =  f   \,\, \textrm{ on } \,\, \partial D.$$
Assuming that $k$ is not an associated Neumann eigenvalue for the differential operator $\Delta +k^2 n$ in $D$ we have that the source problem \eqref{source}  is well-posed(see for e.g. \cite{coltonkress,SteknumericsSun}). Therefore, we can define the NtD operator associated with source problem \eqref{source} as $T: L^2(\partial D) \longmapsto L^2(\partial D)$ such that 
\begin{align}
Tf = w \big|_{\partial D} \quad \text{ where } \,\,  w  \in H^1(D) \,\, \text{ solves \eqref{source} for } \,\, f \in L^2(\partial D). \label{NtD} 
\end{align} 
By the Trace Theorem(see for e.g. \cite{evans}) we have that Range$(T) \subseteq H^{1/2}(\partial D)$ and the compact embedding of $H^{1/2}(\partial D)$ into $L^2(\partial D)$(see for e.g. \cite{cakoni-colton}) implies that $T$ is a compact operator. Now, let $\tau \in \C$ be an eigenvalue of $T$ with corresponding eigenfunction $w$, then by \eqref{source} we have that 
$$T w \big|_{\partial D} = - \lambda^{-1}  w \big|_{\partial D} \quad \text{ which implies that } \quad  \tau = - \lambda^{-1}.$$
Note, that $\tau \neq 0$ provided that $k$ is not a Dirichlet eigenvalue for the differential operator $\Delta +k^2 n$ in $D$. 

\begin{assumption}\label{assume}
The wave number $k \in \R$ is not a Dirichlet or Neumann eigenvalue for the differential operator $\Delta +k^2 n$ in $D$. 
\end{assumption}

Notice, that assumption \ref{assume} is not restrictive since the set of Dirichlet or Neumann eigenvalues is discrete which gives that any choice of wavenumber $k$ is almost surely not an associated eigenvalue. Also, if $n$ is complex-valued then there are no real  Dirichlet or Neumann eigenvalues.

\section{Analysis of the Approximation}\label{conv-analysis} 
Here we analyze the proposed approximation method of the variational formulation \eqref{varform} of the inverse scattering Steklov eigenvalue problem. The method proposed here will be referred to as a Neumann-Galerkin method. This is a Galerkin method where the basis functions are taken to be a finite number of Neumann eigenfunctions for the Laplacian. The basis functions are denoted $\phi_j \in H^1(D)$ with the corresponding Neumann eigenvalues $\sigma_j \in \R_{\geq 0}$ that satisfy   
\begin{align}
- \Delta \phi_j = \sigma_j \phi_j \,\, \text{ in } \,\, D,  \quad {\partial_\nu \phi_j} =0  \,\, \textrm{ on } \,\, \partial D \quad \text{ with }  \quad  \| \phi_j\|_{L^2(D)}=1. \label{neumann-eig}
\end{align}
Here, we assume that the sequence $\sigma_j$ is arranged in increasing order. 
 \subsection{Analysis of the Approximation Space}\label{approx}
Now, we will analyze the approximation space given by 
$$V_N (D) = \text{span}\big\{ \phi_j\big\}_{{j=1}}^{N} \quad \text{ for some fixed } \quad N \in \N.$$
We begin by studying the approximation properties of the finite dimensional subspace $V_N (D) \subset H^1(D)$. It is well-known that the eigenfunctions $\{ \phi_j \}_{_{j=1}}^{^\infty}$ form an orthonormal basis of $L^2(D)$ and that for any $f \in H^1(D)$ 
\begin{align}
f = \sum\limits_{j=1}^{\infty} (f,\phi_j)_{L^2(D)} \phi_j \quad \text{ with } \quad \big\|f \big\|^2_{H^1(D)} =  \sum\limits_{j=1}^{\infty} (1+\sigma_j ) \big|(f,\phi_j)_{L^2(D)}\big|^2 < \infty \label{fourier-series}
\end{align}
by the results in Chapter 9 of \cite{team-pancho}. The series representation \eqref{fourier-series} along with Weyl's law will be used to show the approximation rates for the space $V_N (D)$. Recall, Weyl's law(see for e.g. \cite{weyl-law} {\color{black} equation (1.32)} as well as \cite{weyl-law2,weyl-law3}) gives that there exists two constants $c_1 , c_2 >0$ independent of $j$ such that 
 $$c_1 j^{2/d} \leq \sigma_j \leq c_2 j^{2/d}\quad \text{ for } \quad j\gg1 $$
where again the dimension $d=2,3$. Now, we define the $L^2(D)$ projection onto the approximation space $V_N (D)$  by $\Pi_N :L^2(D) \longmapsto V_N (D)$ such that 
 $$\Pi_N f =  \sum\limits_{j=1}^{N}  (f,\phi_j)_{L^2(D)} \phi_j \quad \text{for some fixed} \quad N \in \N$$ 
 for all $f \in L^2(D)$. By \eqref{fourier-series} we have the {\color{black}norm} convergence
 $$\big\| (I-\Pi_N)f \big\|^2_{H^1(D)} =  \sum\limits_{j=N+1}^{\infty} (1+\sigma_j ) \big|(f,\phi_j)_{L^2(D)}\big|^2 \to 0 \quad \text{as} \quad N \to \infty$$
 for any $f \in H^1(D)$. We now prove some convergence rates in the approximation space $V_N (D)$. This will give that the approximation space has sufficient approximation properties for our Galerkin method.  For the rest of the paper $C$ will be an arbitrary positive constant that does not depend on parameter $N \in \N$.

\begin{theorem}\label{L2convrate}
For any $f \in H^1(D)$ we have the estimate 
$$\big\| (I-\Pi_N) f \big\|_{L^2(D)} \leq \frac{C}{(N+1)^{1/d}} \| f \|_{H^1(D)} \quad \text{ as} \quad N \to \infty.$$
\end{theorem} 
\begin{proof}
By the definition of the projection operator $\Pi_N$ we have that 
\begin{align*}
\big\| (I-\Pi_N) f \big\|^2_{L^2(D)}  &= \sum\limits_{j=N+1}^{\infty}  \big|(f,\phi_j)_{L^2(D)}\big|^2 \\
						   &\leq \sigma_{N+1}^{-1} \sum\limits_{j=N+1}^{\infty} \sigma_j \big|(f,\phi_j)_{L^2(D)}\big|^2 \\
						   &\leq \frac{C}{(N+1)^{2/d}} \sum\limits_{j=1}^{\infty} (1+\sigma_j) \big|(f,\phi_j)_{L^2(D)}\big|^2						   
 \end{align*}
 provided that $N$ is large enough. Note that we have used Weyl's law for the Neumann eigenvalues. This proves the result by \eqref{fourier-series}. 
 \end{proof}

 \begin{theorem}\label{H2convrate}
For any $f \in H^2(D)$ such that $\partial_{\nu} f = 0$ on $\partial D$ we have the estimate 
$$\big\| (I-\Pi_N) f \big\|_{H^1(D)} \leq \frac{C}{(N+1)^{1/d}} \| f \|_{H^2(D)} \quad \text{ as} \quad N \to \infty.$$
\end{theorem} 
\begin{proof} 
To begin, we notice that $\Delta f \in L^2(D)$ and since $\{ \phi_j \}_{_{j=1}}^{^\infty}$ is an orthonormal basis of $L^2(D)$ we have that 
$$ \Delta f = \sum\limits_{j=1}^{\infty} (\Delta f,\phi_j)_{L^2(D)} \phi_j. $$
By Green's Second Theorem we derive that 
 $$(\Delta f,\phi_j)_{L^2(D)}= ( f,\Delta \phi_j)_{L^2(D)} = - \sigma_j ( f, \phi_j)_{L^2(D)}$$
 where we have used \eqref{neumann-eig} as well as the zero Neumann condition for $f$. Therefore, we can conclude that 
 $$ \|\Delta f \|^2_{L^2(D)} = \sum\limits_{j=1}^{\infty} \sigma_j^{2} \big|(f,\phi_j)_{L^2(D)}\big|^2 <\infty. $$
Now, as in the previous result we will use Weyl's law for the Neumann eigenvalues. To this end, we estimate 
\begin{align*}
\big\| (I-\Pi_N) f \big\|^2_{H^1(D)}  &\leq  2 \sum\limits_{j=N+1}^{\infty}  \sigma_j \big|(f,\phi_j)_{L^2(D)}\big|^2 \\
						   &\leq 2 \sigma_{N+1}^{-1} \sum\limits_{j=N+1}^{\infty} \sigma_j^2 \big|(f,\phi_j)_{L^2(D)}\big|^2 \\
						   &\leq \frac{C}{(N+1)^{2/d}} \sum\limits_{j=1}^{\infty} \sigma_j^2 \big|(f,\phi_j)_{L^2(D)}\big|^2						   
 \end{align*}
 provide that $N$ is large enough. This proves the claim. 
 \end{proof}
 
Being motivated by the proof of Theorem \ref{H2convrate} we define a subset of $L^2(D)$ denoted by $\mathscr{D} ( \Delta^m )$ for some $m \in \R_{\geq 0}$ such that 
\begin{align}
\| f \|^2_{\mathscr{D} ( \Delta^m )} = \sum\limits_{j=1}^{\infty} \sigma_j^{2m} \big|(f,\phi_j)_{L^2(D)}\big|^2 <\infty. \label{laplacian-m}
\end{align} 
It is clear that $\mathscr{D} \big( \Delta^m \big)$ is a Hilbert space with norm given by equation \eqref{laplacian-m}. If $m$ is a positive integer this subspace of $L^2(D)$ can be seen as the space of functions where the $m$th Laplacian applied to the series representation \eqref{fourier-series} is a convergent series in $L^2(D)$.  We now prove a convergence rate for $f \in H^1(D) \cap \mathscr{D} \big( \Delta^m \big)$ for $m > 1/2$. 

\begin{theorem}\label{specconvrate}
For any $f \in H^1(D) \cap \mathscr{D} \big( \Delta^m \big)$ such that $m > 1/2$ we have the estimate 
$$\big\| (I-\Pi_N) f \big\|_{H^1(D)} \leq \frac{C}{(N+1)^{(2m-1)/d}} \| f \|_{\mathscr{D} ( \Delta^m )} \quad \text{ as} \quad N \to \infty.$$
\end{theorem} 
\begin{proof}
To prove the claim, recall we have that for $N$ sufficiently large 
 $$\big\| (I-\Pi_N) f \big\|^2_{H^1(D)} \leq 2 \sum\limits_{j=N+1}^{\infty} \sigma_j \big|(f,\phi_j)_{L^2(D)}\big|^2 \leq \frac{2}{{\sigma_{N+1}^{(2m-1)}}}\sum\limits_{j=N+1}^{\infty} \sigma_j^{2m} \big|(f,\phi_j)_{L^2(D)}\big|^2$$
 where we have used the series representation. Now, by again appealing to Weyl's law we conclude that 
 $$\big\| (I-\Pi_N)f \big\|^2_{H^1(D)}  \leq \frac{C}{(N+1)^{2(2m-1)/d}}\sum\limits_{j=1}^{\infty} \sigma_j^{2m} \big|(f,\phi_j)_{L^2(D)}\big|^2 $$
 proving the estimate by the definition of $ \mathscr{D} \big( \Delta^m \big)$.
 \end{proof}

\subsection{Analysis of the Spectral Approximation}\label{conv-eig}
Here we prove the convergence and error estimates for the Neumann-Galerkin approximation method for computing the inverse scattering Steklov eigenvalues. The analysis in this section uses the approximation properties of the space $V_N(D)$. 

To begin, let the trace space of $V_N(D)$ be denoted 
$$V_N(\partial D) = \big\{ f_N \in L^2(\partial D) \, : \, f_N = w_N |_{\partial D} \,\, \text{where} \,\, w_N \in V_N(D) \big\} \subset L^2(\partial D).$$
We now define the Neumann-Galerkin approximation of the NtD mapping as the operator $T_N :L^2(\partial D) \longmapsto V_N(\partial D)$ such that  $w_N \in V_N (D)$ satisfies 
\begin{align}\label{spec-source}
a(w_N ,\varphi_N )= b( f ,\varphi_N)  \quad \text{ for all } \,\,\, \varphi_N \in V_N (D)  \quad \text{ where } \quad T_N f = w_N \big|_{\partial D} .
\end{align} 
It is clear that if \eqref{spec-source} is well-posed then $T_N$ is a well defined compact operator. The goal now is to prove the well-posedness of the discrete source problem \eqref{spec-source}. 

In \cite{SteknumericsSun} it is shown that the sesquilinear form $a( \cdot \, ,\cdot ) + \alpha \,  (\cdot \, , \cdot)_{L^2(D)}$ is coercive on $H^1(D)$
 for $\alpha>0$ sufficiently large. This implies that \eqref{spec-source} is Fredholm of index zero by appealing to the compact embedding of {\color{black}$H^1(D)$ into $L^2(D)$}. Therefore, we can conclude that uniqueness implies well-posedness. We now use a duality argument to prove uniqueness. To this end, we define $u \in H^1(D)$ to be the unique solution to 
$$\Delta u+k^2 \overline{n} u = (w_N-w)  \,\, \textrm{ in } \,\, {D} \quad \text{and} \quad {\partial_\nu u} = 0  \,\, \textrm{ on } \,\, \partial D$$
where $w \in H^1(D)$ is the solution to \eqref{source} and $w \in V_N(D)$ is a solution to  \eqref{spec-source}. By elliptic regularity(see for e.g. \cite{evans} {\color{black}page 334}) we have that the solution $u \in H^2(D)$ and satisfies the regularity estimate 
$$ \| u\|_{H^2(D)} \leq C \| w-w_N \|_{L^2(D)}.$$
Therefore, by appealing to Green's First Theorem we have that 
\begin{align*}
\| w-w_N \|^2_{L^2(D)} & = a(w-w_N ,u ) \\
				   & = a \big(w-w_N , u -\Pi_N u \big) \quad \text{by Galerkin orthogonality} \\
			           &\leq C \| w-w_N \|_{H^1(D)} \| u -\Pi_N u \|_{H^1(D)}.
 \end{align*}
By using Theorem \ref{H2convrate} and the regularity estimate we have that			           
\begin{align}\label{source-estimate1}
\| w-w_N \|_{L^2(D)} \leq \frac{C}{(N+1)^{1/d}} \| w-w_N \|_{H^1(D)}.
\end{align}
Now, using the fact that $a( \cdot \, ,\cdot ) + \alpha \,  (\cdot \, , \cdot)_{L^2(D)} $ is coercive on $H^1(D)$ along with the  Galerkin orthogonality and inequality  \eqref{source-estimate1} we have the estimates 
\begin{align*}
\| w-w_N \|^2_{H^1(D)}&  \leq C \big| a(w-w_N ,w-w_N ) +\alpha \| w-w_N \|^2_{L^2(D)} \big| \\
				   & = C \big| a(w-w_N ,w-\Pi_N w ) +\alpha \| w-w_N \|^2_{L^2(D)} \big|  \\
			           &\leq C \| w-w_N \|_{H^1(D)} \| w -\Pi_N w \|_{H^1(D)}  + \frac{C}{(N+1)^{2/d}} \| w-w_N \|^2_{H^1(D)}.
 \end{align*}
This implies that for $N$ sufficiently large we have the estimate 
\begin{align}\label{source-estimate2}
\| w-w_N \|_{H^1(D)} \leq C \| w -\Pi_N w \|_{H^1(D)}.
\end{align}
Now, if we let the source $f=0$ then the well-posedness of \eqref{source} implies that $w=0$ and we conclude that $w_N=0$ due to inequality \eqref{source-estimate2}. This implies that the approximation of the NtD mapping $T_N$ is a well define compact operator. 

Now, define the Neumann-Galerkin approximation of the inverse scattering Steklov eigenvalue problem \eqref{varform} to be given by: find the values $ \lambda_N \in \C$ and nontrivial $w_N \in V_N(D)$ that satisfies the variational equality  
\begin{align}
a(w_N ,\varphi_N )= - \lambda_N b(w_N,\varphi_N)  \quad \text{ for all } \,\,\, \varphi_N \in V_N (D). \label{spec-varform} 
\end{align} 
Therefore, just as in Section \ref{problem-statement} we have that $\lambda_N \neq 0$ satisfies \eqref{spec-varform} provided that 
$$T_N w_N \big|_{\partial D} = - \lambda^{-1}_N  w_N \big|_{\partial D} \quad \text{ where } \quad \| w_N\|_{L^2(\partial D)} =1.$$
In order to prove the convergence of the approximation we can use the classical results in \cite{babuska-osborn,osborn}. To this end, we are now ready to prove that the approximation $T_N$ converges to $T$ in norm. 

\begin{theorem}\label{NtDconv}
 Let the operators $T: L^2( \partial D) \longmapsto L^2( \partial D)$ be as defined in \eqref{NtD} and $T_N :L^2( \partial D) \longmapsto V_N (\partial D)$  be as defined in \eqref{spec-source}. Then  
 $$\big\|T -T_N \big\|_{L^2( \partial D) \mapsto L^2( \partial D)} \leq \frac{C}{(N+1)^{1/2d}} \quad \text{ as} \quad N \to \infty.$$
\end{theorem} 
\begin{proof}
To prove the claim, we have that for any $f \in L^2( \partial D)$ that 
\begin{align*}
\big\|(T -T_N) f \big\|_{L^2( \partial D)}&  = \| w-w_N \|_{L^2(\partial D)} \\
				   			 & \leq C \| w-w_N \|^{1/2}_{L^2(D)} \| w-w_N \|^{1/2}_{H^1(D)} \quad \text{by Theorem 1.6.6 in \cite{FEM-book}}  \\
			           			 & \leq \frac{C}{(N+1)^{1/2d}} \| w-w_N \|_{H^1(D)} \quad \text{by inequality \eqref{source-estimate1}}\\
						 	 & \leq \frac{C}{(N+1)^{1/2d}} \| w -\Pi_N w \|_{H^1(D)} \quad \text{by inequality \eqref{source-estimate2}.}
\end{align*}
Now, by the uniform boundedness principle(see for e.g. \cite{numerics-book}) we have that the operator norm of $I - \Pi_N : H^1(D) \longmapsto H^1(D)$  is bounded uniformly with respect to $N$. Therefore, we obtain
$$\big\|(T -T_N) f \big\|_{L^2( \partial D)} \leq\frac{C}{(N+1)^{1/2d}} \| w \|_{H^1(D)}\leq\frac{C}{(N+1)^{1/2d}} \| f \big\|_{L^2( \partial D)}$$
by the well-posedness of \eqref{source}, proving the claim. 
\end{proof}

By the norm convergence of the approximation of the NtD mapping we can conclude that convergence of the approximation of the Steklov eigenvalues and functions in $L^2(\partial D)$. The following is a consequence of the results in \cite{osborn} and Theorem \ref{NtDconv}. 

\begin{theorem} \label{eig-conv}
Let $(\lambda_N , w_N) \in \C \times V_N (D)$ be an eigenpair for \eqref{spec-varform}. Then there is an eigenpair $(\lambda, w)  \in \C \times H^1 (D)$ for \eqref{varform} such that 
$$|\lambda -\lambda_N | \leq\frac{C}{(N+1)^{1/2d}} \,\,\,  \text{ and } \,\,\,  \|w - w_N \|_{L^2(\partial D)} \leq\frac{C}{(N+1)^{1/2d}} \quad \text{ as} \quad N \to \infty.$$ 
\end{theorem}

Theorem \ref{eig-conv} gives the convergence of our approximation. We are now interested in determining a spectral convergence rate for our approximation. To do so, we denote the eigenspace associated with $\lambda$ as $E(\lambda)$ as a subset of $H^1(D)$ and recall the space $\mathscr{D} \big( \Delta^m \big)$ defined by the series constraint \eqref{laplacian-m}. {\color{black} In the next result we will assume that $E(\lambda) \subset \mathscr{D} \big( \Delta^m \big)$ for some $m > 1/2$. The assumption that the eigenspace be a subset of $\mathscr{D} \big( \Delta^m \big)$ can be seen as a  constraint on the decay of the Fourier coefficients. Indeed, we have that $w \in \mathscr{D} \big( \Delta^m \big)$ if and only if $|(w,\phi_j)_{L^2(D)}| =o\big( j^{-p} \big)$ as $j \to \infty$ for $p=(1+4m)/{2d}$. Recall, that the faster the Fourier coefficients decay the smoother the function by the M-Test provided that $\phi_j$ are smooth functions which is the case when $D$ has a smooth boundary. Therefore, the assumption that $E(\lambda) \subset \mathscr{D} \big( \Delta^m \big)$ can also be seen as a regularity constraint on the eigenfunctions. }

\begin{theorem} \label{eig-convrate}
Assume the eigenspace $E(\lambda) \subset \mathscr{D} \big( \Delta^m \big)$ for some $m > 1/2$. Then for every eigenvalue $\lambda_N$ for \eqref{spec-varform}  there is an eigenvalue  $\lambda$ for \eqref{varform} such that 
$$\big|\lambda -\lambda_N \big| \leq \frac{C}{ (N+1)^{(4m-1)/2d}}\sup\limits_{w \in E(\lambda) \, : \,  \|w\|_{L^2(\partial D)} =1} \|w \|_{\mathscr{D} ( \Delta^m )} \quad \text{ as} \quad N \to \infty.$$  
\end{theorem}
\begin{proof}
To prove the estimate we use Theorem 7.3 in \cite{babuska-osborn}. From this we have that we need to estimate 
$$\big\|T -T_N \big\|_{E(\lambda) \mapsto L^2( \partial D)} =  \sup\limits_{w \in E(\lambda) \, :\,  \|w\|_{L^2(\partial D)} =1}\big\|(T -T_N) w \big\|_{L^2( \partial D)} $$
 in order to obtain the convergence rate for the eigenvalues. For any $w \in E(\lambda)$ we have that $\partial_\nu w = - \lambda w$ on $\partial D$. Therefore, we have the estimates 
\begin{align*}
\big\|(T -T_N) w \big\|_{L^2( \partial D)}& \leq \frac{C}{(N+1)^{1/2d}} \| (I-\Pi_N) w \|_{H^1(D)} \quad \text{by the proof of Theorem \ref{NtDconv}}\\
							 & \leq \frac{C}{(N+1)^{(4m-1)/2d}} \| w \|_{\mathscr{D} ( \Delta^m )}  \quad \text{by Theorem \ref{specconvrate}}.
\end{align*}
Now, by taking the supremum over $E(\lambda)$ such that $\| w \|_{L^2(\partial D)}=1$ we obtain that  
$$\big\|T -T_N \big\|_{E(\lambda) \mapsto L^2( \partial D)} \leq \frac{C}{ (N+1)^{(4m-1)/2d}} \sup\limits_{w \in E(\lambda) \, : \,  \|w\|_{L^2(\partial D)} =1} \|w \|_{\mathscr{D} ( \Delta^m )}$$
which proves the claim. 
\end{proof}

Even though our main focus is on computing the eigenvalues we will consider the convergence of the eigenfunctions in the region $D$. The following result gives the convergence of the eigenfunctions in the $H^1(D)$ norm. 

\begin{theorem} \label{eigfunc-conv}
Let $w_N \in V_N(D)$ be an eigenfunction for \eqref{spec-varform}. Then there is an eigenfunction $w \in H^1(D)$ for \eqref{varform} such that 
$$ \|  w - w_N \|_{H^1(D)} \leq\frac{C}{(N+1)^{1/4d}} \quad \text{ as} \quad N \to \infty.$$
\end{theorem}
\begin{proof}
To prove the claim, we first note that by Theorem \ref{NtDconv} we have the estimates  
$$  \|  w - w_N \|_{L^2(\partial D)} \leq\frac{C}{(N+1)^{1/2d}} \quad \text{and} \quad  |\lambda -\lambda_N | \leq\frac{C}{(N+1)^{1/2d}}.$$
Since, $w_N \in H^1(D)$ satisfies \eqref{spec-source} with $f = -\lambda_N w_N$ the well-posedness of \eqref{spec-source} and converges estimates above implies that $w_N$ is a {\color{black} bounded} sequence in $H^1(D)$. Now, to prove the convergence we again use a duality argument. Therefore, we let $u \in H^1(D)$ be the unique solution to 
$$\Delta u+k^2 \overline{n} u = (w_N-w)  \,\, \textrm{ in } \,\, {D} \quad \text{and} \quad {\partial_\nu u} = 0  \,\, \textrm{ on } \,\, \partial D.$$
By elliptic regularity we have that $u \in H^2(D)$ and is bounded with respect to $N$. Green's First Theorem and some simple calculations using \eqref{varform} and \eqref{spec-varform} gives that 
\begin{align*}
\| w-w_N \|^2_{L^2(D)} & = a(w-w_N ,u ) \\
				   &=(\lambda_N - \lambda) b (w, u) + \lambda_N b (w - w_N, u) \\
				   & \hspace{2cm}- \lambda_N b \big( w_N , (I-\Pi_N)u \big) - a \big( w_N , (I-\Pi_N)u \big).
 \end{align*}
Notice, that by the convergence rate of the eigenfunctions on $\partial D$ and eigenvalues we have the estimate 
$$\big| (\lambda_N - \lambda) b (w, u) \big| \leq\frac{C}{(N+1)^{1/2d}} \quad \text{and} \quad  \big| \lambda_N b (w - w_N, u) \big| \leq\frac{C}{(N+1)^{1/2d}}$$
where we have also used the fact that $u$ is {\color{black} bounded} in $H^2(D)$. By appealing to the approximation rate in Theorem \ref{H2convrate} 
$$\Big| \lambda_N b \big( w_N , (I-\Pi_N)u \big) \Big| \leq\frac{C}{(N+1)^{1/d}} \quad \text{and} \quad  \Big|a \big( w_N , (I-\Pi_N)u \big)\Big| \leq\frac{C}{(N+1)^{1/d}}$$
where we have used the Trace Theorem and the fact that $w_N$ is a {\color{black} bounded} sequence in $H^1(D)$. This implies that we have the $L^2(D)$ convergence estimate 
$$ \|  w - w_N \|^2_{L^2(D)} \leq\frac{C}{(N+1)^{1/2d}}.$$
Simple calculations give that for $(\lambda , w)$ and $(\lambda_N , w_N)$ eigenpairs for \eqref{varform} and \eqref{spec-varform} then 
$$a\big(w_N-w , w_N -w\big) + \lambda b\big(w_N-w , w_N -w\big) = \big(\lambda-\lambda_N  \big)b\big(w_N , w_N\big).$$
Recall, that the sesquilinear form $a( \cdot \, ,\cdot ) + \alpha \,  (\cdot \, , \cdot)_{L^2(D)}$ is coercive on $H^1(D)$ for $\alpha >0$ sufficiently large.
Therefore, we have that  
\begin{align*}
\| w-w_N \|^2_{H^1(D)} &\leq C \big| a(w-w_N ,w-w_N ) +\alpha \| w-w_N \|^2_{L^2(D)} \big| \\
				    & \leq C \Big\{ \big| \lambda-\lambda_N \big|  +  \big| \lambda b\big(w_N-w , w_N -w\big)\big|  + \| w-w_N \|^2_{L^2(D)} \Big\}. 
 \end{align*}
By combining the above estimates proves the claim.  
\end{proof}

\section{Numerical Examples}\label{numerics}   
This section is dedicated to providing numerical examples of our approximation method for computing the inverse scattering Steklov eigenvalues. The convergence will be studied for constant and variable refractive index $n$. We also consider the inverse spectral problem of estimating/recovering the refractive index from the knowledge of the eigenvalues. This problem is also considered in \cite{Stekinvproblem} where a Bayesian approach is used. Here we will use the monotonicity(see for e.g. \cite{Aniso-stekloff,Stekinvproblem}) of the largest positive eigenvalue denoted $\lambda_1$ to estimate a positive refractive index. 

We take the domain $D$ to be given by the unit disk in $\R^2$ to compare with separation of variables. Note that $D$ can always be chosen to be a disk or square that is sufficiently large such that $\Omega \subseteq D$. In the following examples, the approximation space is given by the span of finitely many Neumann eigenfunctions 
$$\phi_{j}(r,\vartheta) = \text{J}_{p} \left(\sqrt{\sigma_{p,q}} \, r\right) \cos(p \vartheta)\quad \text{with index} \quad j=j(p,q) \in \N.$$
The square root of the Neumann eigenvalues $\sqrt{\sigma_{p,q}}$ corresponds to the $q$th non-negative root of the $p$th first kind Bessel function derivative denoted $\text{J}'_{p}$ for all $p\in \N \cup\{0\}$ and $q \in \N$. Some of the values of $\sqrt{\sigma_{p,q}}$ can be found in \cite{zeroes}. 

We will use $25$ basis functions  where $0\leq p \leq 4$ and $1\leq q \leq 5$.
In the following sections we take the approximation space 
$$ V_N (D) \subseteq \text{Span} \Big\{ \phi_{j(p,q)}(r,\vartheta)  \Big\}_{p=0\, ,\, q=1}^{p=4 \, ,\, q=5} \quad \text{ giving  that } \quad w_N (x) = \sum\limits_{j=1}^{N} c_j \phi_j (x) $$
for constants $c_j$. 
Substitution $w_N $ into \eqref{spec-varform}  and taking $\varphi_N = \phi_i$  we obtain that the eigenvalues $\lambda_N$ satisfying \eqref{spec-varform} correspond to the eigenvalues for the matrix equation  
\begin{align}\label{g-eig}
\left( {\bf A}+ \lambda_N {\bf B}\right) \vec{c} =0 \quad \text{ where } \quad {\bf A}_{i,j}  = a(\phi_j , \phi_i) \quad \text{ and } \quad {\bf B}_{i,j} =  b(\phi_j , \phi_i).
\end{align}
By appealing to Green's First Theorem we obtain that 
$$ a(\phi_j , \phi_i) = \int\limits_{D} \grad \phi_j \cdot \grad \overline{\phi_i} -k^2 n \phi_j \overline{\phi_i} \, \text{d}x= \int\limits_{D} \big(\sigma_j -k^2 n \big) \phi_j \overline{\phi_i} \, \text{d}x.$$
Using the orthogonality of the cosines representing the angular part of the $\phi_j$ on $\partial D$ to reduce the computational cost for $b(\phi_j , \phi_i)$. Also, notice that by appealing to the $L^2(D)$ orthogonality of $\phi_j$ this methods becomes very cost effective when the refractive index is constant in $D$. 
For the examples presented this method is implemented in \texttt{MATLAB} where the `\texttt{eig}' command is used to solve \ref{g-eig}. To compute the Galerkin matrices we employ a 2d Gaussian quadrature scheme. 

\subsection{Comparison to Separation {\color{black}of} Variables}\label{eig-approx}

In this section, we will compare our approximation to the analytically computed eigenvalue for the unit disk. To do so, assuming that $\Omega=D$ is given by the unit disk in $\R^2$ then in \cite{Stek-invscat} we have that for $n$ constant that the eigenvalues can be determined by separation of variables. This gives that  
\begin{align}
\lambda = -k \sqrt{n} \frac{ \text{J}'_{m}(k \sqrt{n}) }{  \text{J}_{m}(k\sqrt{n}) } \quad \text{ for any $m\geq0$. } \label{sov-formula}
\end{align}

We will test the accuracy of the approximation by comparing it to the values given by \eqref{sov-formula}. In our examples, we will take $n$ to be real and complex-valued constant to show that the approximation is valid for either case. Also, for all our numerical experiments in this and the following section, we will take the wavenumber $k=1$ for simplicity. In Tables \ref{conv1} and \ref{conv2}, we present the approximated eigenvalue $\lambda_{1,N}$ for various degrees of freedom $N$ as well as the relative error. 

\begin{table}[ht!]
\centering  
\begin{tabular}{  c | c | c | c | c  } 
\hline                  
       &  $N=10$  &  $N=15$ & $N=20$ & $N=25$ \\ [0.5ex] 
\hline                  
\hline                  
 $\lambda_{1,N}$ & $1.1872162$  & $1.2500365$   & $1.2816379$   & $1.3007182$ \\ [0.5ex]
Rel. Error              & $0.1378900$  & $0.0922724$   & $0.0693246$   & $0.0554693$ \\ [0.5ex]
\hline 
\end{tabular}
\caption{The first approximated eigenvalue for various $N$ with $n=2$ to demonstrate the convergence to $\lambda_1 =1.3771053$ as $N \to \infty$. }\label{conv1}
\end{table}

\begin{table}[ht!]
\centering  
\begin{tabular}{  c | c | c | c | c  } 
\hline                  
       & $N=10$  &  $N=15$ & $N=20$ & $N=25$ \\ [0.5ex] 
\hline                  
\hline                  
 $\lambda_{1,N}$ &  
    $\begin{array}{l} 1.10178 \\  \,\,\, +0.70628\text{i} \end{array}$  
 & $\begin{array}{l} 1.12973 \\  \,\,\, +0.77689\text{i} \end{array}$   
 & $\begin{array}{l} 1.14240 \\  \,\,\, +0.81262\text{i} \end{array}$   
 & $\begin{array}{l} 1.14957 \\  \,\,\, +0.83424\text{i} \end{array}$ \\ [2ex]
Rel. Error             &  $0.1519793$  & $0.1011953$   & $0.0758263$   & $0.0605793$ \\ [0.5ex]
\hline 
\end{tabular}
\caption{The first approximated eigenvalue for various $N$ with $n=2+\text{i}$ to demonstrate the convergence to $\lambda_1 =1.17422 + 0.92123\text{i}$ as $N \to \infty$.  }\label{conv2}
\end{table}

{\color{black} In Figure \ref{convplots}, the log-log convergence plots for the eigenvalues are presented which gives a convergence rate $\mathcal{O}(N^{-1})$ in the two examples. This would seem to suggest that the eigenfunctions are in $\mathscr{D} \big( \Delta^m \big)$ for $m=5/4$ by the convergence result in Theorem \ref{eig-convrate}. This give that the Fourier coefficients for the eigenfunction are $o(j^{-3/4})$ as $j \to \infty$. Due to the fact that $\phi_j$ are uniformly bounded in $D$ this implies that the eigenfunction is continuous in $D$ by the M-Test. }

\begin{figure}[ht!]
\centering
{\includegraphics[width=7.5cm]{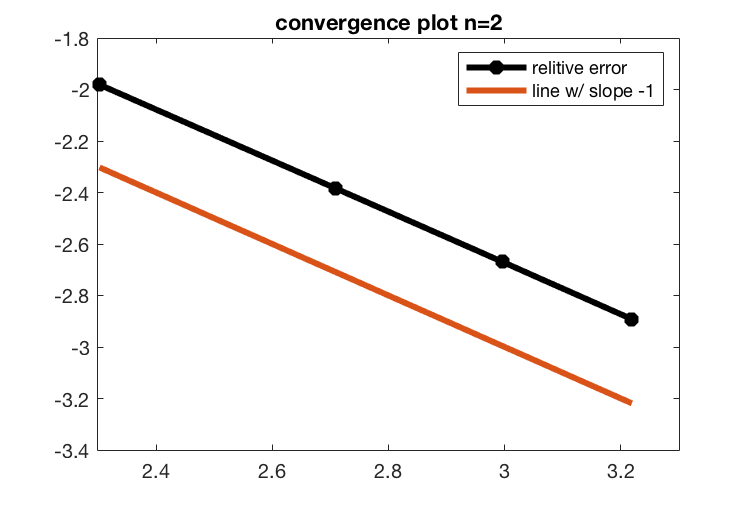}} \hspace{-0.5cm}
{\includegraphics[width=7.5cm]{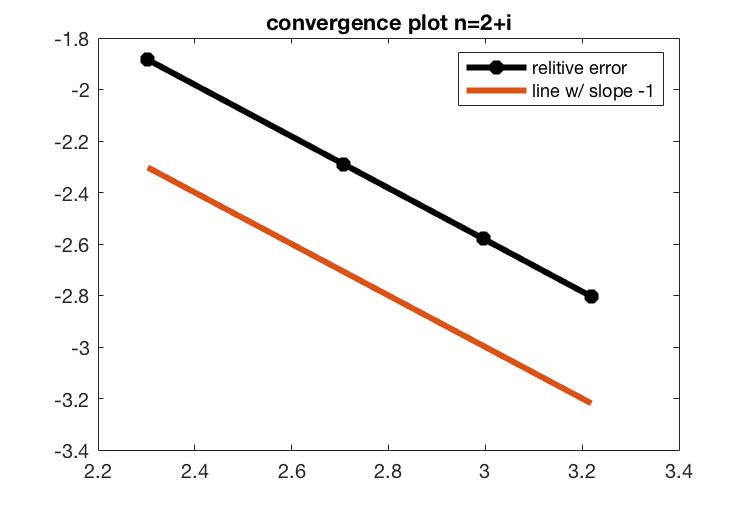}} 
\caption{Convergence plots of the first eigenvalues in the unit disk with the refractive indices $n=2$ and $n=2+\text{i}$. We compare with a line of slope $-1$ where we see 1st order convergence as $N \to \infty$.} \label{convplots}
\end{figure}

We will now show that our numerical scheme is valid for a piecewise constant refractive index in $D$. To this end, assume that $D$ is the unit disk and the scatterer $\Omega$ is given by the disk with radius $0<\rho < 1$. Now, define the refractive index  
$$n =  \begin{cases} \,\, 1, \quad \quad \quad\quad \,\,  \rho <r  \\ \,\, (1+n_1)^2, \quad  r \leq \rho \end{cases}$$
where $n_1$ is a positive constant. In \cite{Stek-invscat} it is shown using separation of variables and the asymptotic expansions of the Bessel function's that the first eigenvalue is given by the following expansion for $k\rho \ll 1$
\begin{align}
\lambda_1 = -  k  \frac{ \text{J}'_0(k) }{  \text{J}_0(k) } +  \frac{1}{2} n_1 (2+n_1)(k\rho)^2 + \mathcal{O} \left((k \rho)^4 \right). \label{asym-formula}
\end{align}
Using this expansion for the first eigenvalue we will compare with our numerically approximated eigenvalue. 
In Table \ref{compare3}, we report the eigenvalues computed for $\rho=1/2^{p}$ by the approximation with $N=25$ as well as the values from the first two terms in the asymptotic expansion \eqref{asym-formula} {\color{black}and the exact eigenvalues computed via separation of variables} for various values of $p$. Notice, that our approximation is valid {\color{black}for this example of a piecewise constant $n$} when $\rho \ll 1$ with $N=25$ where as a finite element method would require a large amount of degrees of freedom to assure accuracy. 

\begin{table}[ht!]
\centering  
\begin{tabular}{  c | c | c | c  } 
\hline                  
          & Approximation  $\lambda_{1,N}$  &  Asymptotic Formula &  {\color{black}Exact Eigenvalue}    \\ [0.5ex] 
\hline                  
\hline                  
   $p=1$&  $0.780984210069194$    &   $0.700080915004306$   &   $0.763513625502361$ \\  [0.5ex]
   $p=2$&  $0.617530179557115$    &   $0.606330915004306$   &   $0.615333593156268$\\  [0.5ex]
   $p=3$&  $0.581111365230462$     &   $0.582893415004306$   &   $0.584679376860770$\\  [0.5ex]
   $p=4$&  $0.565820243626941$    &   $0.577034040004306$   &   $0.577444105677795$ \\  [0.5ex]
\hline 
\end{tabular}
\caption{Comparison with the asymptotic formula \eqref{asym-formula} {\color{black}and the exact eigenvalue} for $n=2$ where the scatterer is given by the disk with $\rho=1/2^{p}$ for $p=1,2,3,4$. }\label{compare3}
\end{table}

\subsection{Parameter Estimation}\label{param-approx}
In this section, we provide a new algorithm for estimating the (real-valued) refractive index from the knowledge of the inverse scattering Steklov eigenvalues. It has been shown in \cite{Aniso-stekloff,Stek-invscat} that the Steklov eigenvalues can be recovered from the knowledge of the far-field data via the Linear Sampling Method and Generalized Linear Sampling Method. In \cite{Stekinvproblem} the eigenvalues are recovered from near-field measurement by using the Reciprocity Gap Method. Therefore, for simplicity, we will use the eigenvalues computed by the Neumann-Galerkin approximation as a stand-in for the eigenvalues computed from the data and we wish to estimate $n$. 

To begin, we present the numerical approximation of the Steklov eigenvalues and eigenfunctions for a variable refractive index $n$ {\color{black}with $N=25$}. The eigenvalues presented here will be used in the approximation of $n$. In Table \ref{variable}, we report the first three eigenvalues where the scatterer is either the unit disk and disk with radius $\rho=1/2$. We take the refractive index $n=2+r(\sin\theta-\cos\theta)$ in both cases. Since we have also proven the convergence of the eigenfunctions, we also provided the contour plots for the first three eigenfunctions associated with the eigenvalues in Figure \ref{eigfunc1}.

\begin{table}[H]
\centering  
\begin{tabular}{  c | c | c | c  } 
\hline                  
Disk w/ radius $\rho$  & 1st eigenvalue $\lambda_{1,N}$  & 2nd eigenvalue $\lambda_{2,N}$  & 3rd eigenvalue  $\lambda_{3,N}$  \\ [0.5ex] 
\hline                  
\hline                  
$\rho=1$     & $1.33947280348$ &   $-0.47739381775$   & $-1.75712435055$  \\ [0.5ex]
$\rho=1/2$  & $0.78174886356$ &   $-0.74001156781$   & $-1.95378594455$  \\  [0.5ex]
\hline 
\end{tabular}
\caption{The first three eigenvalues for two different scatterers where the refractive index is given by $n=2+r(\sin\theta-\cos\theta)$. }\label{variable}
\end{table}

\begin{figure}[ht!]
\centering
{\includegraphics[width=4.8cm]{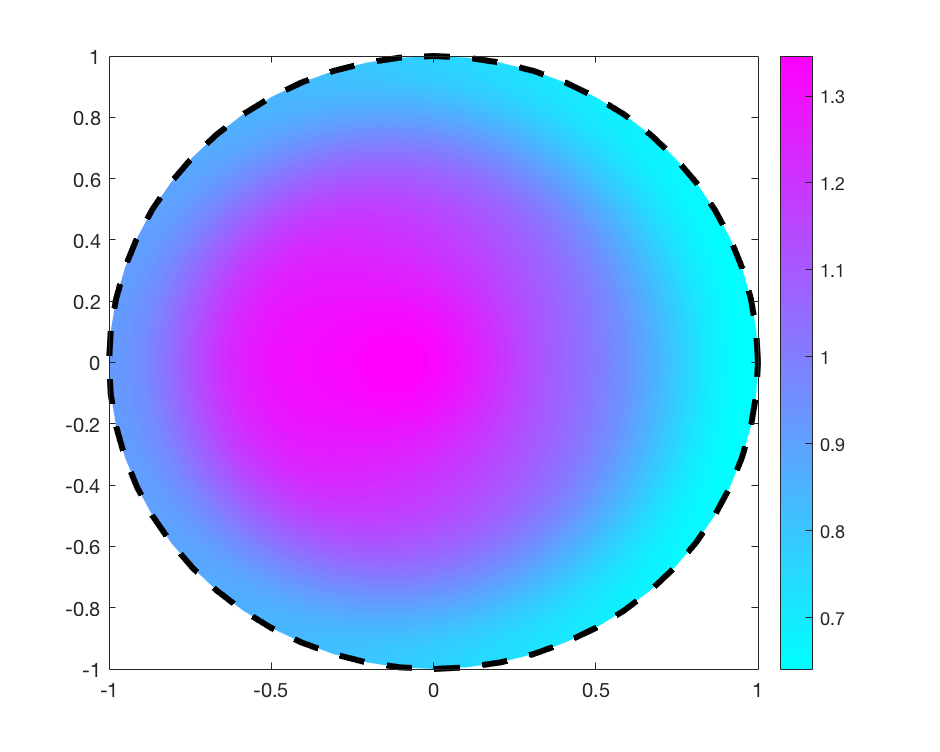}} 
{\includegraphics[width=4.8cm]{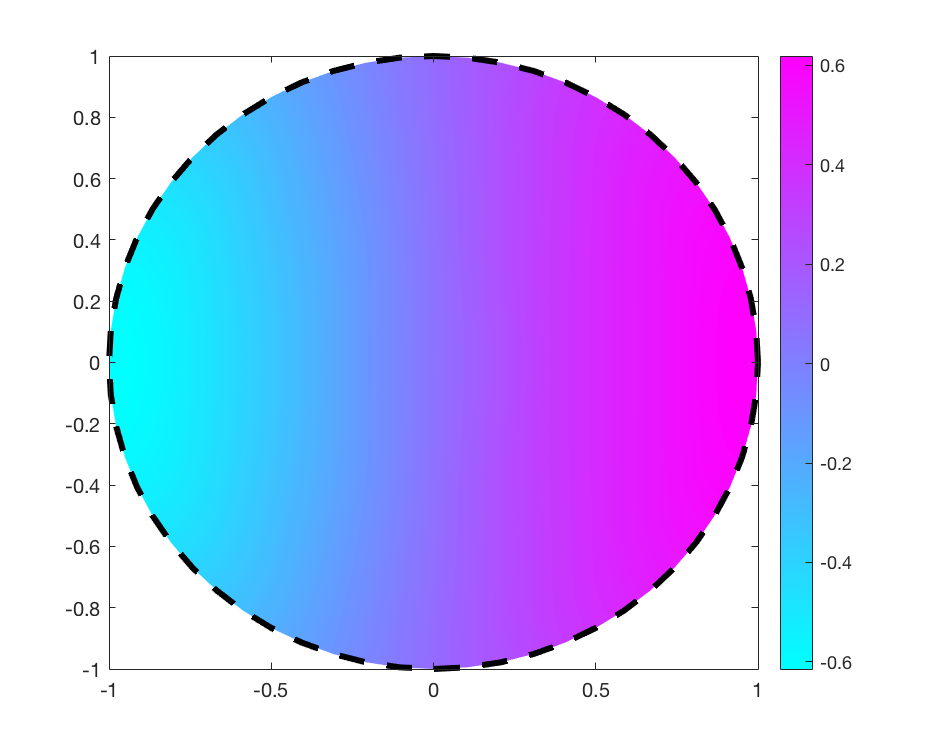}} 
{\includegraphics[width=4.8cm]{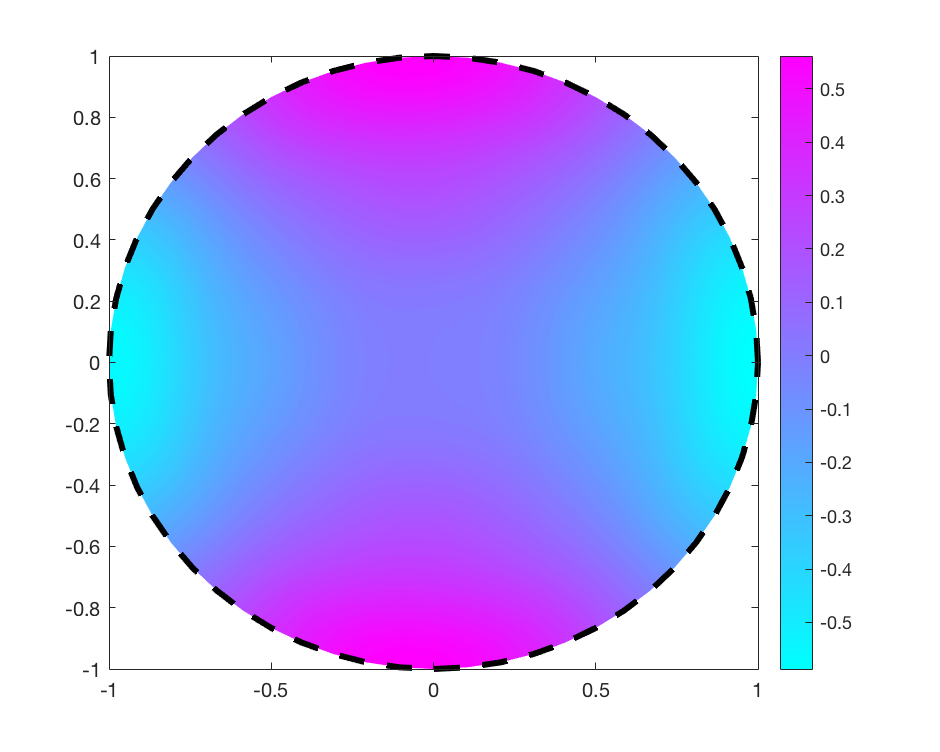}}\\
{\includegraphics[width=4.8cm]{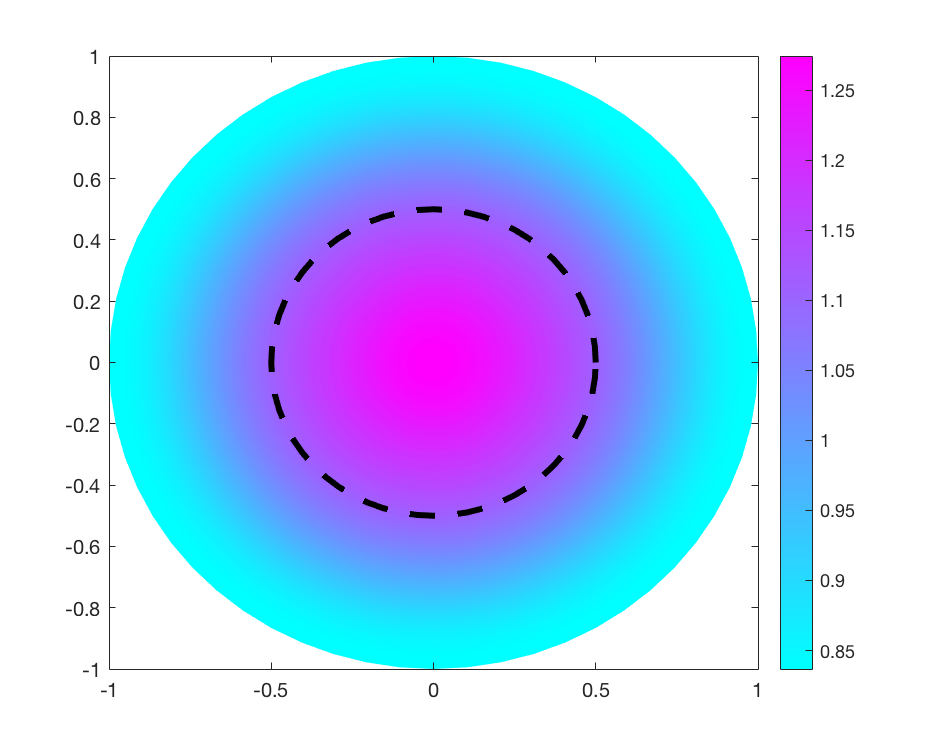}} 
{\includegraphics[width=4.8cm]{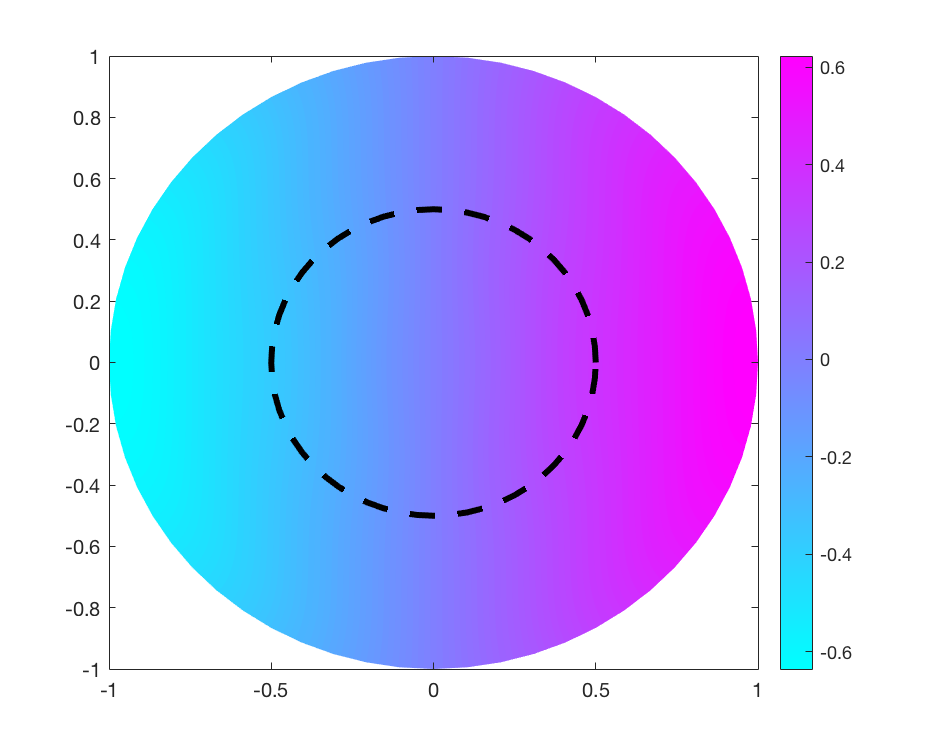}} 
{\includegraphics[width=4.8cm]{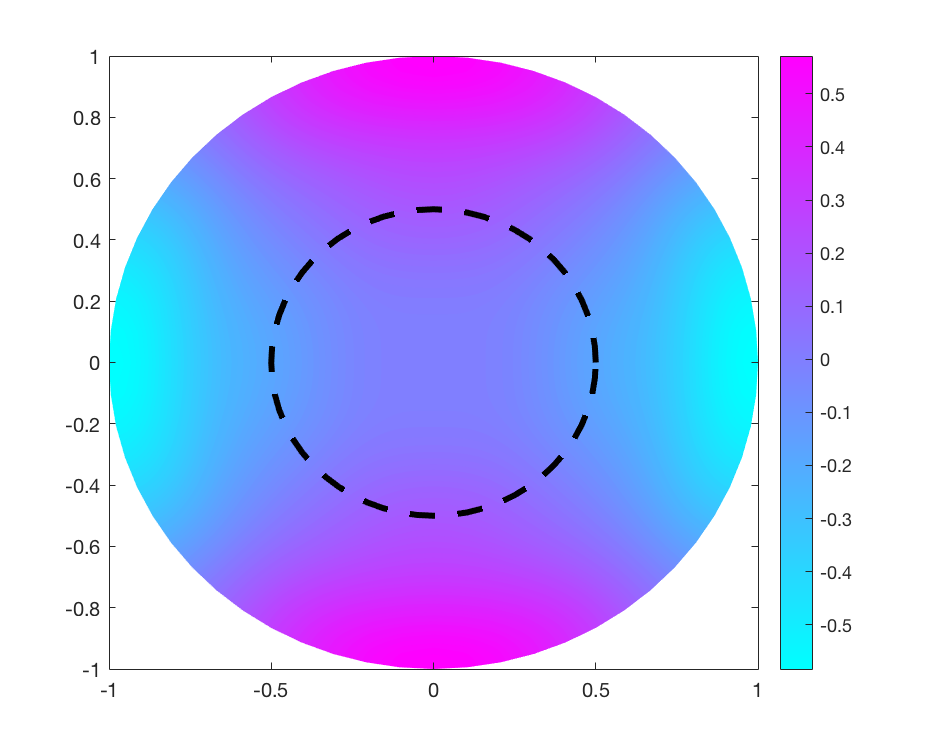}}
\caption{Plots of the first three eigenfunctions for the unit disk and disk with radius $1/2$ scatterers with refractive index $n=2+r(\sin\theta-\cos\theta)$. The dotted line is the boundary of the scatterer.} \label{eigfunc1}
\end{figure}

Now, we report the approximated eigenvalues for a constant refractive index where the scatterer $\Omega \neq D$ {\color{black}for $N=25$}. We consider the boundary of the scatterer to be given in polar coordinates such that 
$$\partial \Omega = \rho(\theta) (\cos\theta, \sin\theta )$$ 
with $0<\rho(\theta) < 1$ is a $2\pi$-periodic function. Here we consider a pear, elliptical, and rounded-square shaped scatterer given by 
\begin{eqnarray*}
\rho(\theta) &=& 0.3 (2+0.3\cos(3\theta) ), \\
\rho(\theta) &=& 0.35 (2+0.3\sin(2\theta) ) \quad \text{and,} \\
\rho(\theta) &=& 0.75 \big( |\sin(\theta)|^5 +  |\cos(\theta)|^5 \big)^{-1/5} 
\end{eqnarray*}
respectively. The eigenvalues are reported in Table \ref{pw} and the associated eigenfunctions are plotted in Figure \ref{eigfunc2}.

\begin{table}[ht!]
\centering  
\begin{tabular}{  c | c | c | c } 
\hline                  
Scatterer  & 1st eigenvalue $\lambda_{1,N}$  & 2nd eigenvalue $\lambda_{2,N}$  & 3rd eigenvalue  $\lambda_{3,N}$  \\ [0.5ex] 
\hline                  
\hline                  
Pear-Shaped &      $0.89339093521$ &   $-0.70841945488$   & $-1.94018366846$  \\[0.5ex]
Elliptical-Shaped & $0.97880829577$ &   $-0.67854111485$   & $-1.93207985011$  \\[0.5ex]
Rounded-Square & $1.11759427187$ &   $-0.60744622788$   & $-1.90328635229$  \\[0.5ex]
\hline 
\end{tabular}
\caption{ The first three eigenvalues for three different scatterers with $n=2$. }\label{pw}
\end{table}

\begin{figure}[ht!]
\centering
{\includegraphics[width=4.8cm]{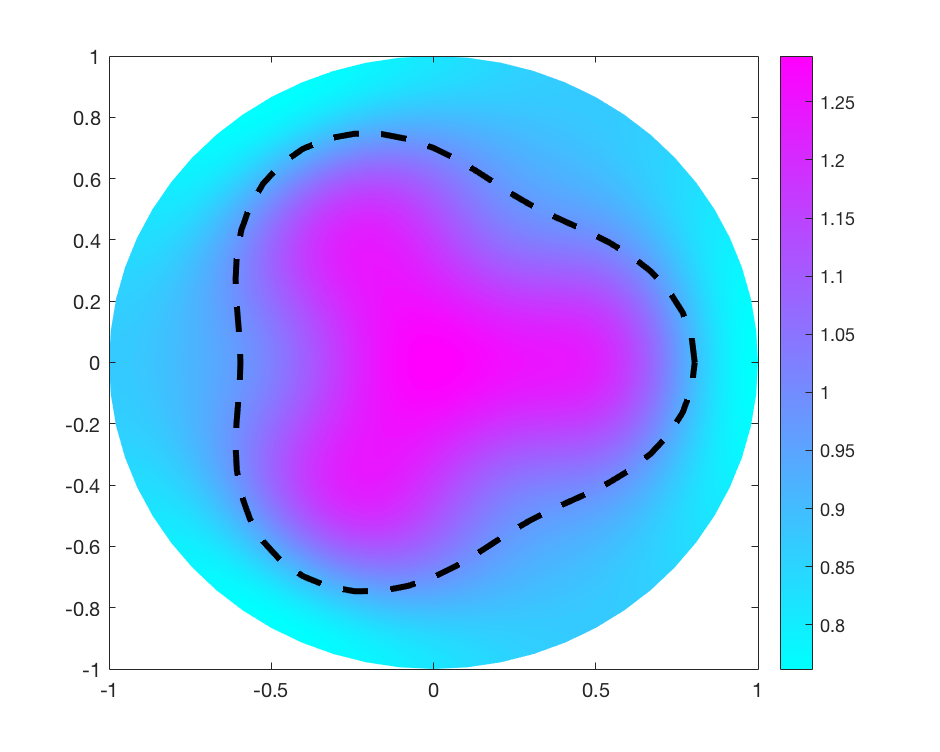}} 
{\includegraphics[width=4.8cm]{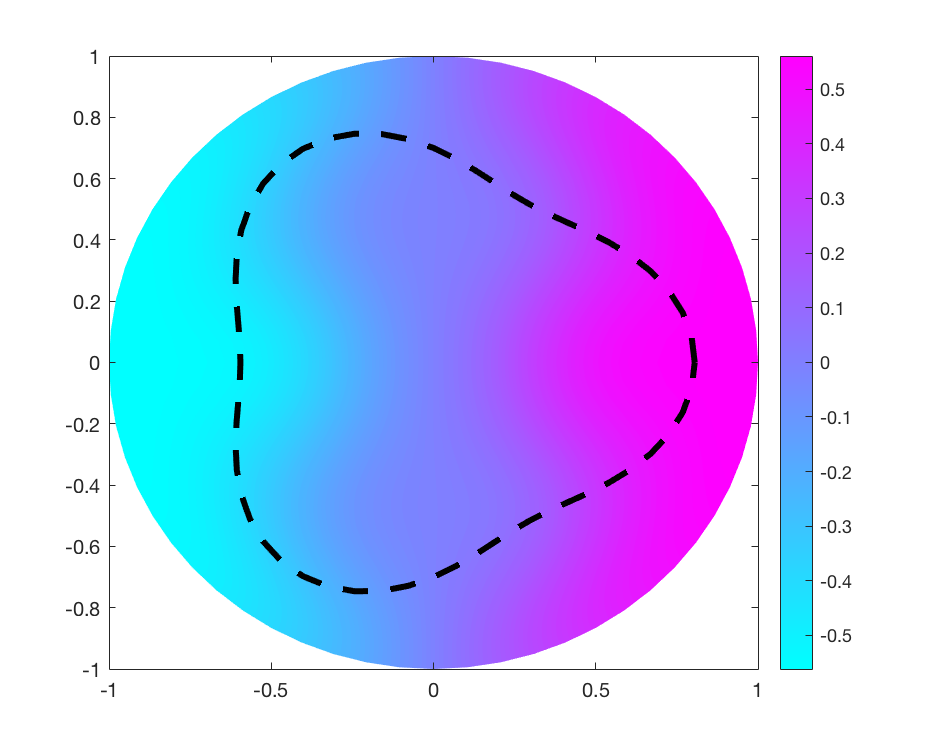}} 
{\includegraphics[width=4.8cm]{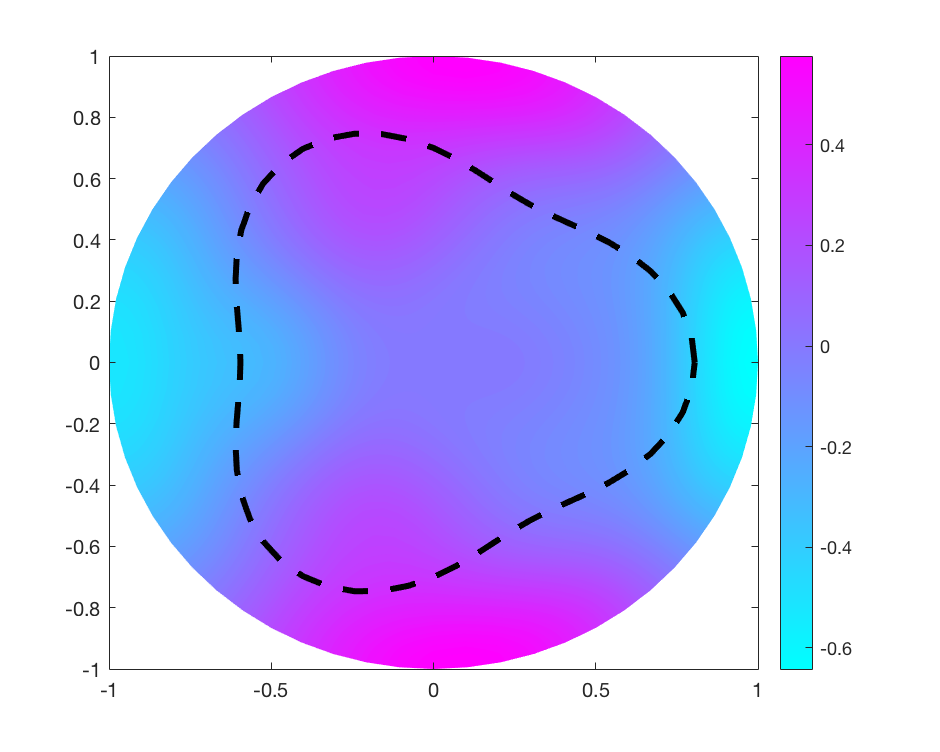}}\\
{\includegraphics[width=4.8cm]{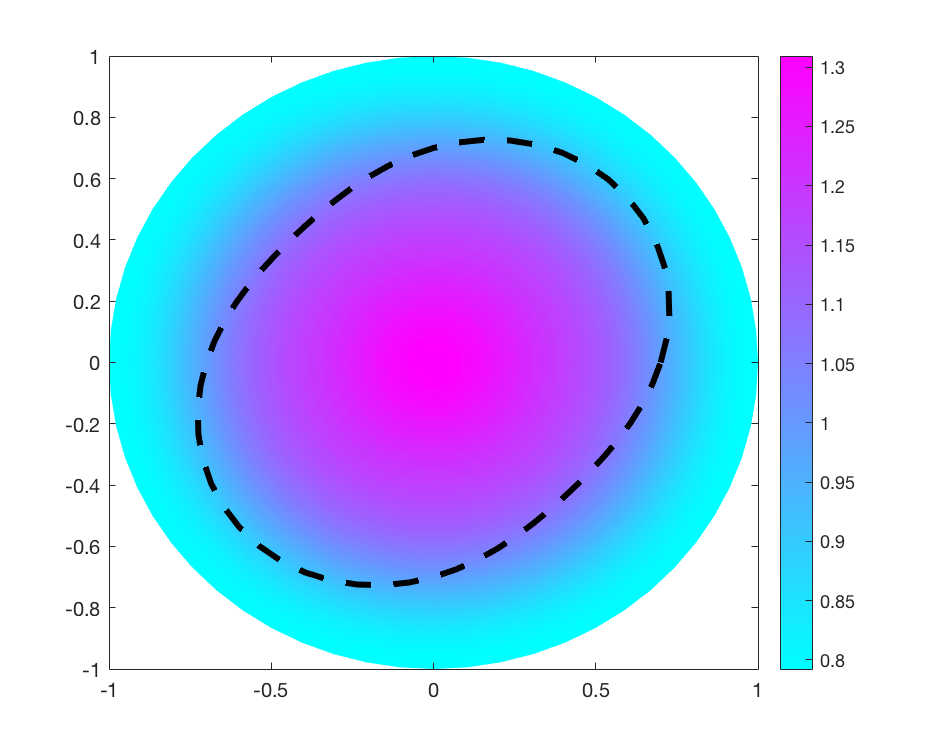}} 
{\includegraphics[width=4.8cm]{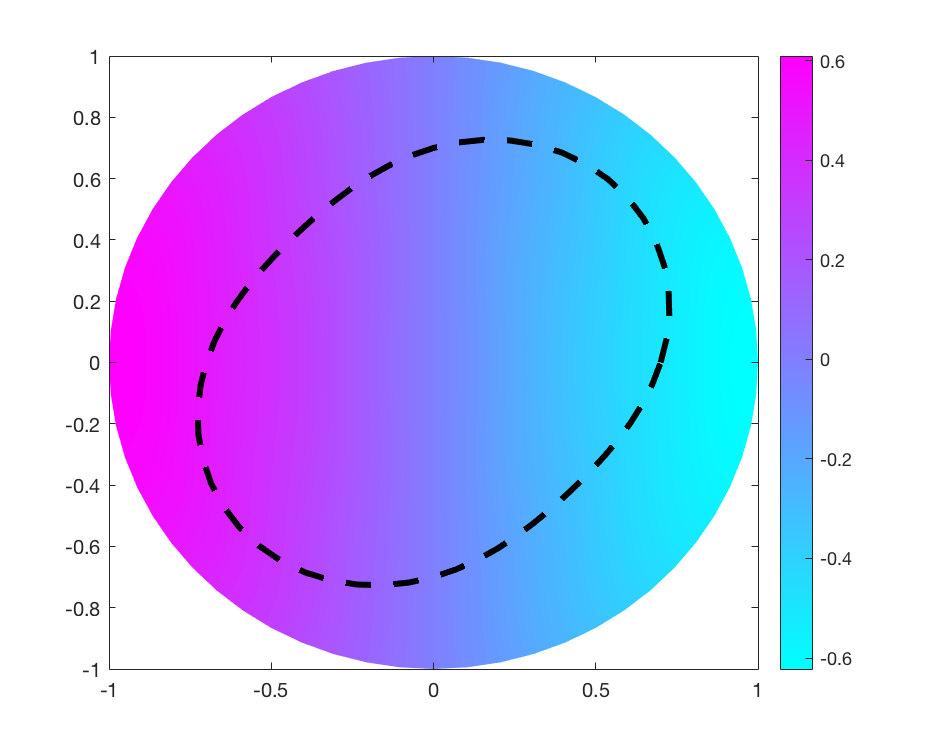}} 
{\includegraphics[width=4.8cm]{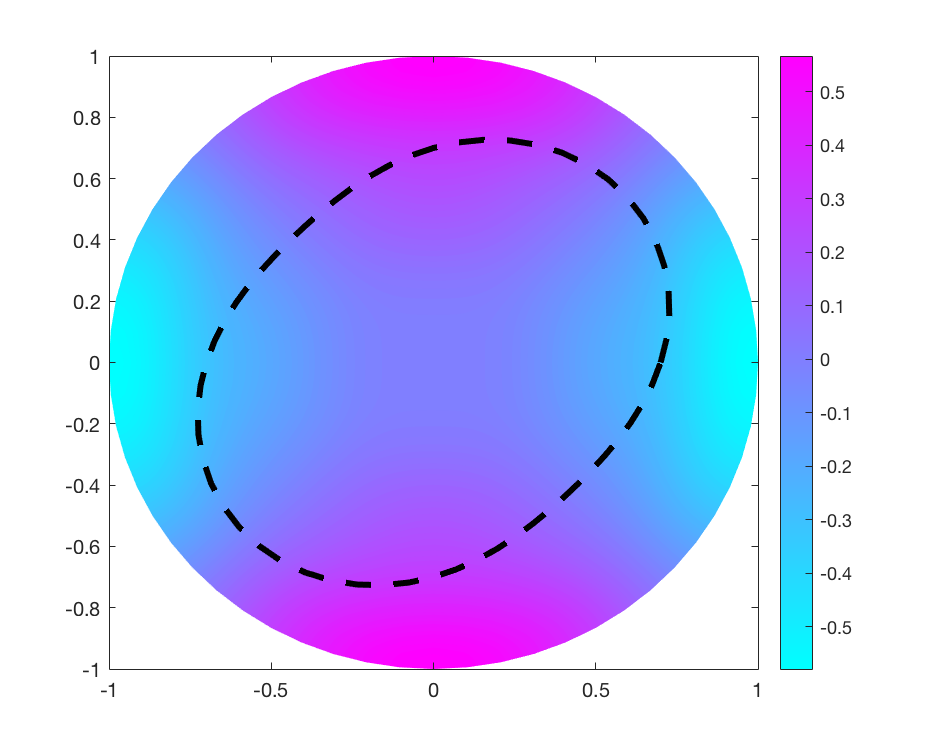}}\\
{\includegraphics[width=4.8cm]{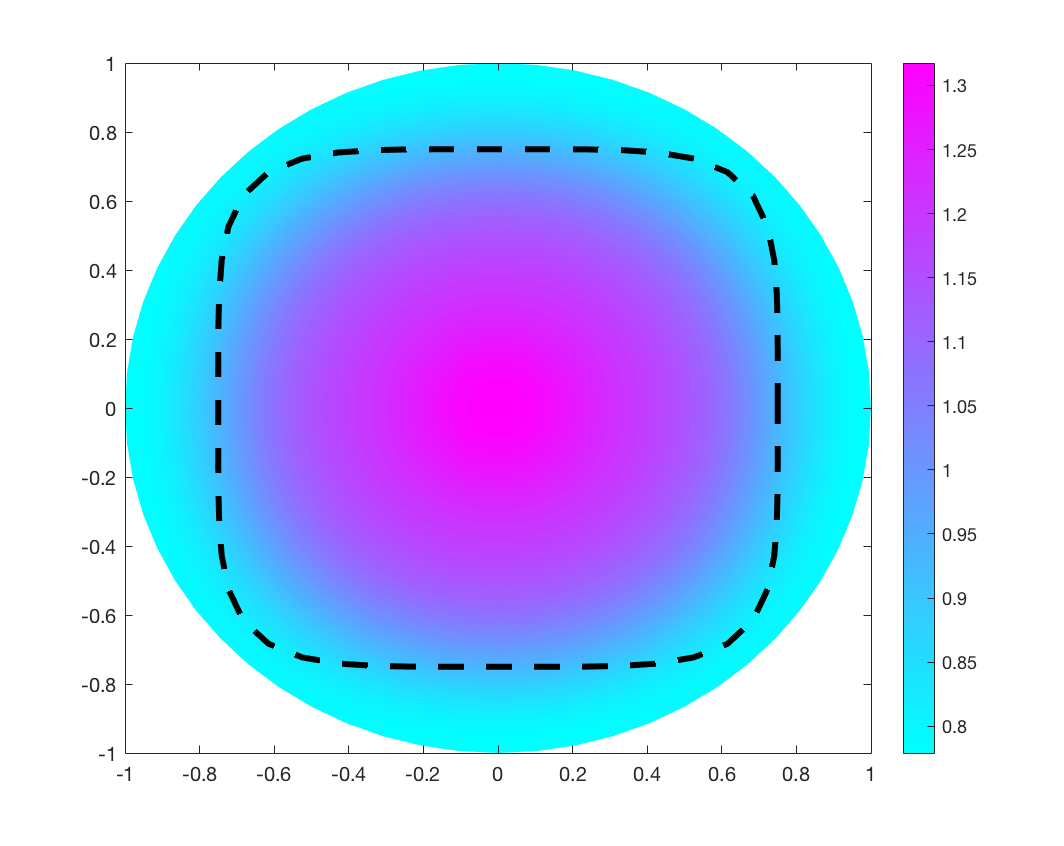}} 
{\includegraphics[width=4.8cm]{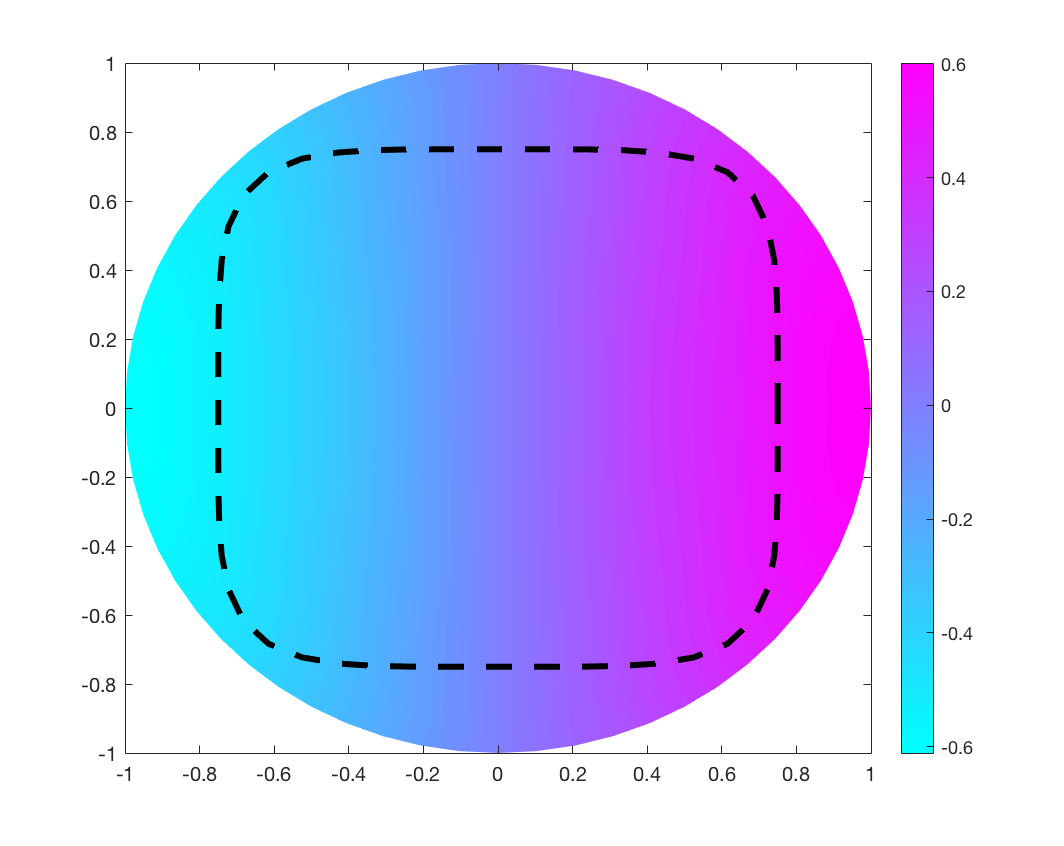}} 
{\includegraphics[width=4.8cm]{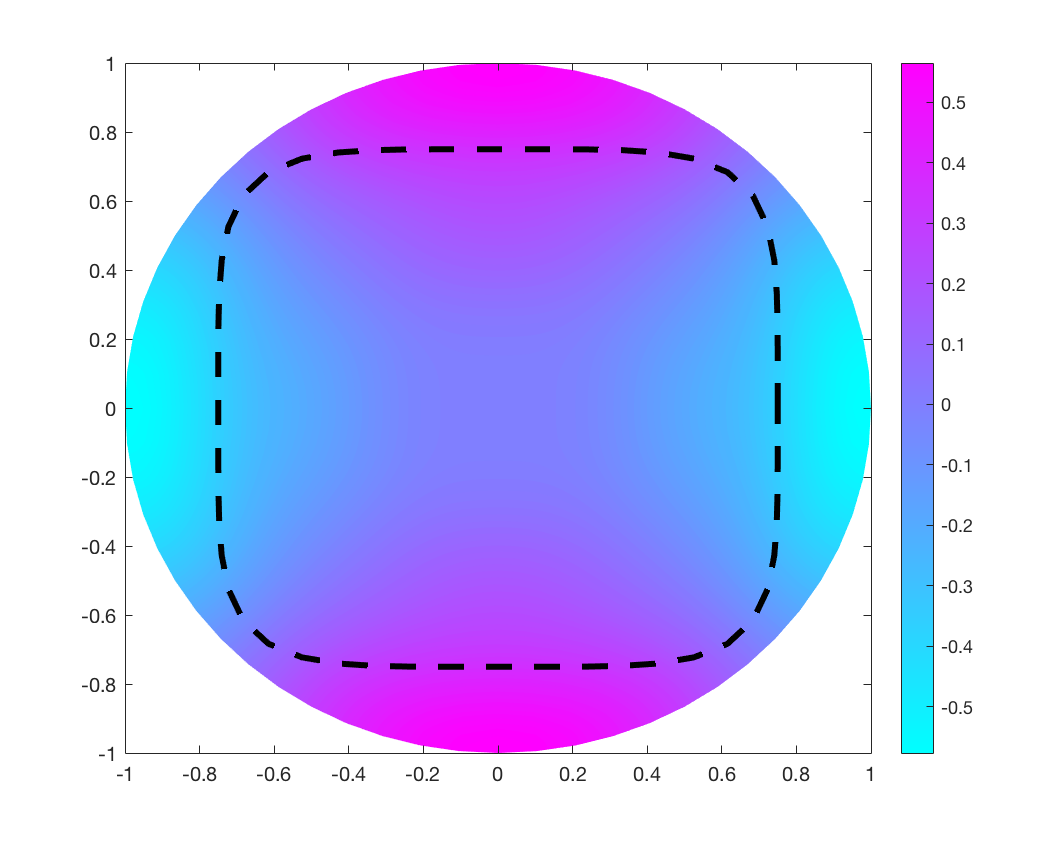}}
\caption{Plots of the first three eigenfunctions for the pear, elliptical and rounded-square shaped scatterer with $n=2$. The dotted line is the boundary of the scatterer.} \label{eigfunc2}
\end{figure}


Lastly, we turn our attention to estimating the refractive index. To this end, we will assume that the scatterer $\Omega$ is known and begin with the case when $D=\Omega$. The method proposed here is to approximate $n$ by a positive constant. Therefore, in order to approximate $n$ we find the unique value $n_{\text{approx}}>0$ satisfying 
\begin{align} \label{findn1}
\lambda_{1}(n_{\text{approx}}) = \lambda_{1,N}(n) 
\end{align}
where  $\lambda_{1}$ is given by equation \eqref{sov-formula} for $m=0$.
 To solve the above transcendental equation \eqref{findn1} we use the `\texttt{fzero}' command in \texttt{MATLAB}. As we see in our examples the approximation $n_{\text{approx}}$ seems to be the average value of $n$ in $D$ just as the case for the transmission eigenvalues \cite{te-homog,ziteHarris}. Therefore, we will assume that the solution to \eqref{findn1} approximates the average value of $n$ over $D$. Since we know a priori that $n=1$ in $D \setminus \overline{\Omega}$ we can use a two step process to estimate $n$ when $\Omega \neq D$. 
\begin{itemize}
\item  Step 1: Solve \eqref{findn1} to determine an initial $n_{\text{approx}}>0$. 
\item  Step 2: Define the new approximation $n_{\text{approx},2}$ such that $n = 1$ for $x \in D \setminus \overline{\Omega}$ and $n=n_{\text{approx},2}$ for $x \in{\Omega}$ where the constant $n_{\text{approx},2}$ is given by  
\begin{align} \label{findn2}
n_{\text{approx},2} = \frac{ n_{\text{approx}}\big|D\big| -\big|D \setminus \overline{\Omega} \big| }{\big|{\Omega} \big|}.
\end{align}
\end{itemize} 
Here $| \cdot |$ denotes the area of a Lebesgue measurable set in $\R^2$. Equation \eqref{findn2} is obtained by the assumption that the initial estimate $n_{\text{approx}}$ is the average value of $n$ in $D$. This method is implemented for the eigenvalues presented in Tables \ref{variable} and \ref{pw} where the approximations of the refractive index $n$ are reported in Table \ref{recon}. {\color{black} Here we see that $n_{\text{approx},2}$ approximates the average value of the refractive index $n$ in the scatterer $\Omega$ as one would expect just as in case of using the transmission eigenvalues.}

\begin{table}[ht!]
\centering  
\begin{tabular}{  c | c | c   } 
\hline                  
      Scatterer    & Refractive Index $n$  &  Approximation $n_{\text{approx},2}$    \\ [0.5ex] 
\hline                  
\hline                  
  Disk w/ $\rho=1$      &  $n=2+r(\sin\theta-\cos\theta)$    &   $1.961032$   \\  [0.5ex]
  Disk w/ $\rho=1/2$   &  $n=2+r(\sin\theta-\cos\theta)$    &   $2.181511$   \\  [0.5ex]
  Disk w/ $\rho=1$      &  $n=2$     				      &   $1.920193$   \\  [0.5ex]
  Pear-Shaped            &  $n=2$     				      &   $1.894312$   \\  [0.5ex]
  Elliptical-Shaped	  &  $n=2$  				      &   $2.111828$   \\  [0.5ex]
  Rounded-Square      &  $n=2$    				      &   $2.053623$   \\  [0.5ex]

\hline 
\end{tabular}
\caption{Here we approximate the refractive index $n$ by a constant in the scatterer $\Omega$ for multiple shapes. {\color{black} The average value in $\Omega$ for each refractive index $n$ is equal two where we see that $n_{\text{approx},2}$ approximates the average value in $\Omega$.}}\label{recon}
\end{table}

\section{Summary and Conclusions}\label{last}
In conclusion, we have provided a numerical method for computing the inverse acoustic scattering Steklov eigenvalues via the Neumann spectral-Galerkin approximation method. The approximation space is taken to be the span of the first $N$ Neumann eigenfunctions for the Laplacian in $D$. The analysis presented here is valid for any chosen auxiliary domain $D$ with a  piece-wise smooth boundary $\partial D$ for $d=2$, 3. In particular, the domain $D$ can also be taken as a square/cube centered at the origin to reduce the computational cost of evaluating Bessel functions. One needs to have computed the Neumann eigenfunction for the domain $D$ to employ this method. In the application of inverse scattering that is the focus of this paper, the domain $D$ can be chosen to be a disk that contains the scatterer. Since the Neumann eigenfunctions for a disk or square are well known via separation for variables this method can be always be applied for this problem. In our examples, we see that the approximation is still accurate for a modest size system even when the scatterer $\Omega$ is small in comparison to $D$ which is not the case for the finite element method. We have also presented numerical examples to investigate estimating the refractive index from the first eigenvalue. Another possible application of this method is to use the Neumann spectral-Galerkin method to compute the inverse scattering Trace Class Stekloff eigenvalues studied in \cite{Stektraceclass}. This is a new modified Stekloff eigenvalue problem whose numerical approximation by Galerkin methods has not been rigorously analyzed. Also, for multiple scatterers, one can try and augment the method presented in \cite{gpinvtev} to recover the refractive index.


\end{document}